\documentclass[10pt,twoside,final]{amsart}

\usepackage[english]{babel}
\usepackage{graphicx,epstopdf,epsfig}
\usepackage{amsfonts,epsfig,fancyhdr,graphics,amsmath,amssymb}

\title{On the traces of the product of two linear similarity classes}

\newtheorem{theorem}{Theorem}[section]

\newtheorem{lemma}[theorem]{Lemma}

\newtheorem{remark}[theorem]{Remark}

\newcommand {\PSL }{\mathrm{PSL}}
\newcommand {\mr }{\mathrm{mr}}
\newcommand {\tr }{\operatorname{tr}}

\newcommand {\GL }{\mathrm{GL}}
\newcommand {\SL }{\mathrm{SL}}
\newcommand {\M }{\mathrm{M}}
\newcommand {\Idm }{\mathrm{I}}
\newcommand {\rank }{\operatorname{rank}}
\newcommand {\GF }{\mathrm{GF}}

\usepackage{ifdraft}
\ifdraft{\usepackage[outer]{showlabels}
	\usepackage{todonotes}}{}
\ifdraft{\usepackage[outer]{showlabels}
	\usepackage{todonotes} \usepackage{listlbls} \usepackage{datetime}}{}

\usepackage[pagebackref,colorlinks,citecolor=red,urlcolor=blue,linkcolor=green, bookmarks=false,hypertexnames=true]{hyperref}
\usepackage{orcidlink}
\usepackage{fancyhdr}
\usepackage{verbatim}

\begin{document}
\bibliographystyle{plain}

\setcounter{page}{1}

\thispagestyle{empty}



\author{Klaus Nielsen}\,\orcidlink{0009-0002-7676-2944}
\email{klaus@nielsen-kiel.de}

\ifdraft{\today \ \currenttime}{\date{November 15, 2024}}
\pagestyle{fancy}
\fancyhf{}
\fancyhead[OC]{Klaus Nielsen}
\fancyhead[EC]{Traces of products of similarity classes }
\fancyhead[OR]{\thepage}
\fancyhead[EL]{\thepage}

\maketitle

\begin{abstract}
It is shown that the product of two nonscalar conjugacy  classes of the special linear group $\SL(n,K)$ contains matrices of arbitrary trace if $n \ge 4$ and $K$ is an abitrary field  or $n=3$ and $K$ is finite. 
\end{abstract}

\section{Introduction}

Let $\Omega, \Psi$ be nonscalar similarity classes of $\GL(n, K)$. In \cite{Adan-Bante_Harris}, Adan-Bante and Harris  have shown that the set $\tr \Omega \Psi = \{\tr \omega \psi; \omega \in \Omega, \psi \in \Psi\}$ of traces of
$\Omega \Psi$ satisfies $|\tr \Omega \Psi| \ge q -1$ if $K = \GF(q)$ is a finite field. They also proved that $|\tr\Omega \Psi| \ge \frac{q}{2}$ if $\Omega$ and $\Psi$ are conjugacy classes of $\SL(n,q)$.

It follows that  the product of 2 nontrivial conjugacy classes
of $\PSL(n, q)$
contains at least $q$ conjugacy classes. 
This verifies a conjecture of Arad and Herzog  in a special case.
Arad and Herzog \cite{Arad_Herzog}  conjectured that in a finite simple group $G$ the product of 2 nontrivial conjugacy classes of $G$ is never a conjugacy class.

Let $\M(n, K)$ denote the set of $n \times n$ matrices over the field $K$.
In this note, we want to improve the results of Adan-Bante and Harris and show

\begin{theorem}               \label{theorem-1}
	Let $n \ge 3$.
	Let $\Omega, \Psi$ be similarity classes of $\M(n,K)$. Assume that $\Omega$ and  $\Psi$ are nonscalar.
	Then $\tr(\Omega \Psi)  = K$. 
\end{theorem}

\begin{theorem}               \label{theorem-2}
	Let $n \ge 3$. If $K$ is infinite let $n \ge 4$.
	Let $\Omega, \Psi$ be  conjugacy classes of $\SL(n,K)$. Assume that $\Omega$ and  $\Psi$ are nonscalar.
	Then $\tr \Omega \Psi  = K$. 
\end{theorem}

\section{The traces of the product of 2 similarity classes}

\begin{lemma}               \label{ARADHERZOG1}
	Let $\Omega, \Psi$ be nonscalar similarity classes of $\M(2,K)$. Assume that $\Omega$ is not irreducible.
	\begin{enumerate}
		\item  $\tr \Omega \Psi = K - \{\alpha \tr \Psi\}$ if and only if
		$(\Omega - \alpha \Idm_2)^2 = 0$ and $\Psi$ is irreducible.
		\item $\tr \Omega \Psi = K$ if and only if $\Omega$ is not primary or $\Psi$ is not irreducible.
	\end{enumerate}
\end{lemma}

\begin{proof}
	Let  $\mu(\Omega) = (x - \lambda) (x - \alpha)$, $\mu(\Psi) = x^2 - \gamma x - \delta$. Then   
	\[
	\left (\begin{array} {cc} \alpha & 0\\ \mu & \lambda \end{array} \right )
	\left (\begin{array} {cc} 0  & 1 \\ \gamma & \delta \end{array} \right ) =
	\left (\begin{array} {cc}  0 & \alpha   \\  \lambda \gamma   &  \lambda \delta + \mu \end{array} \right ) \in \Omega \Psi
	\]
	if $\lambda \ne \alpha$ or  $\mu \ne 0$.
	If $\lambda = \alpha$ and $\Psi$ is irreducible, then  $\tr \Omega \Psi = K - \{\alpha \tr \Psi\}$:
	\[
	\left (\begin{array} {cc} \alpha & \mu\\ 0 & \alpha \end{array} \right )
	\left (\begin{array} {cc} \delta - \epsilon & \rho \\ \nu & \epsilon \end{array} \right ) =
	\left (\begin{array} {cc} \alpha (\delta -\epsilon) + \mu \nu & \alpha \rho + \mu \epsilon \\  \alpha \nu   &  \alpha \epsilon  \end{array} \right ).
	\]
\end{proof}

\begin{lemma}        \label{ARADHERZOG2}      
	Let $\Omega, \Psi$ be cyclic similarity classes of $\M(n,K)$. Let $D \in \GL(n-1,K)$. Put $\delta = (\det D)^{-1}\det \Omega \Psi$. Then there exists a row vector 
	$z$ such that 
	\[
	\left (\begin{array} {cc} \delta & z\\ 0 & D \end{array} \right ) \in \Omega \Psi.
		\]
\end{lemma}

\begin{proof}
	We may assume that  $D = L U$, where $L \in \GL(n-1,K)$ is lower, and $U \in \GL(n-1,K)$ is upper triangular. Then there exist row vectors  $x, y$
	such that 
	\[
	\left (\begin{array} {cc} y & \beta\\ L & 0 \end{array} \right ) \in \Omega,
	\left (\begin{array} {cc} 0 & U\\ \gamma & x \end{array} \right ) \in \Psi.
	\]
	Then 
	\[
\left (\begin{array} {cc} \beta \gamma & y U + \beta x\\ 0 & LU \end{array} \right ) =
	\left (\begin{array} {cc} y & \beta\\ L & 0 \end{array} \right ) 
	\left (\begin{array} {cc} 0 & U\\ \gamma & x \end{array} \right ).
	\]
\end{proof}

\begin{lemma}               \label{ARADHERZOG3}
	Let $n \ge 3$.
	Let $\Omega, \Psi$ be nonscalar similarity classes of $\M(n, K)$. 
	Then $\tr \Omega \Psi = K$.
\end{lemma}

\begin{proof}
	We may assume that $r := \partial \mu(\Psi) \le \partial \mu(\Omega)$.
	There exist matrices
	\[
	W = \left (\begin{array} {cc} A & B \\ C & D \end{array} \right ) \in \Omega,
	Q = \left (\begin{array} {cc} R & 0 \\ 0 & S \end{array} \right ) \in \Psi,
	\]
	where $A, R \in \M(r, K)$ are cyclic. Then 
	\[
	\left (\begin{array} {cc} AR^X & BS \\ CR^X & DS \end{array} \right ) \in \Omega \Psi
	\]
	for all $X \in \GL(r, K)$.
	 By lemma \ref{ARADHERZOG2}, $\tr\{ AR^X; X \in \GL(r, K)\} = K$ if $r \ge 3$. So let $r = 2$. 
	In this case, we may assume that $A = 0_1 \oplus \Idm_1$.
	By lemma \ref{ARADHERZOG1}, $\tr \{ AR^X; X \in \GL(r, K)\} = K$.
\end{proof}

\section{The traces of the product of 2 conjugacy classes of $\SL(n,K)$}

We begin with cyclic conjugacy classses.
For a conjugacy class $\Omega \subseteq \SL(n, K)$ let $\overline{\Omega}$ denote the similarity closure of $\Omega$.

\begin{lemma}               \label{ARADHERZOG3a}
	Let $\varphi \in \GL(V)$ be a cyclic linear transformation of a vector space $V$ over $K$.
	Let $n := \dim V  \ge 3$. Then in  suitable basis of $V$
	\[
	\varphi = \left (\begin{array} {cc} 0_m & \Idm_m  \\ C  & \ast
	\end{array} \right ) \mathrm{or} \ 
	\varphi = \left (\begin{array} {ccc} 0_m & 0 & \Idm_m \\ 0 & \beta & \ast \\ C & \ast & \ast
	\end{array} \right ),
	\]
	where $C \in \GL(m, K)$ is cyclic, according to whether $n=2m$ is even or $n=2m+1$ is odd.
\end{lemma}

\begin{proof}
	If $n = 2m$ consider $\varphi$ in a basis $v, v \varphi^2, \dots,  v \varphi^{2m-2},
	v \varphi, v \varphi^3, \dots,  v \varphi^{2m-1}$.
	
	So let $n = 2m+1$ be odd. Consider $\varphi$ in a basis $v, v \varphi^2, \dots, v \varphi^{2m-2}$, $w$, $v \varphi, v \varphi^3, \dots,  v \varphi^{2m-1}$, where $w = v \varphi^{2m} + v$. In this basis, $\varphi$ is similar to a matrix
	\[
	\varphi = \left (\begin{array} {ccc} 0_m & 0 & \Idm_m \\ a & \beta & \ast \\ C & \ast & \ast
	\end{array} \right ).
	\]	
	where $C$ is cyclic and nonsingular. Hence $\varphi$ is similar to
	\[
	\left (\begin{array} {ccc} \Idm_m & 0 & 0 \\ 0 & 1 & -aC^{-1} \\ 0 & 0 & \Idm_m
	\end{array} \right )
	\left (\begin{array} {ccc} 0_m & 0 & \Idm_m \\ a & \beta & \ast \\ C & \ast & \ast
	\end{array} \right )
	\left (\begin{array} {ccc} \Idm_m & 0 & 0 \\ 0 & 1 & aC^{-1} \\ 0 & 0 & \Idm_m
	\end{array} \right )
	= \left (\begin{array} {ccc} 0_m & 0 & \Idm_m \\ 0 & \widetilde{\beta} & \ast \\ C & \ast & \ast
	\end{array} \right ).
	\]	
\end{proof}

\begin{lemma}               \label{ARADHERZOG4}
	Let $m \ge 2$. If $|K| \le 3$ assume that $m \ge 3$.
	Let $\Omega, \Psi$ be cyclic conjugacy  classes of $\SL(2m, K)$. 
	Then $\tr \Omega \Psi = K$.
\end{lemma}

\begin{proof}	
Let $\tau \in K$. By \ref{ARADHERZOG3a}, there exist matrices
\[
W = \left (\begin {array} {cc} 0 & \Idm_m \\ C & D \end{array} \right ) \in \overline{\Omega},
Q = \left (\begin{array} {cc} E & F \\ \Idm_m & 0 \end{array} \right ) \in \overline{\Psi},
\]
where $C, D, E, F \in \M(m,K)$, and $C$ and $F$ are cyclic.

By \ref{ARADHERZOG2}, there exist $R, S \in \GL(m,K)$ such that $C^R F^S = \lambda \Idm_1 \oplus M$, where $M-\lambda \Idm_{m-1}$ is nonsingular. 
Conjugating $W$ with $R\oplus R$ and $Q$ with $S \oplus S$, we may assume that $CF = \lambda \Idm_1 \oplus M$.
There exist $\delta, \epsilon \in K^*$ such that $W^S \in \Omega$ if $\det S = \delta$ and
$Q^T  \in \Psi$ if $\det T = \epsilon$.
By \ref{ARADHERZOG2}, $CF = \Idm_1 \oplus M = Z_1 Z_2$, where
\begin{enumerate}
	\item $Z_1$ and $Z_2$ are cyclic,
	\item $\det Z_1 = \epsilon \delta$, $\tr Z_1  = 0$, and 
	\item $\tr Z_2  = \tau$.
\end{enumerate}
Now let $Z_1 = X Y$, where $\det X = \delta, \det Y = \epsilon$.
Let $S = \Idm_m \oplus X, T = \Idm_m \oplus Y$. Then $W^S \in \Omega, Q^T \in \Psi$ and
\[
W^S Q^T = \left (\begin{array} {cc} 0 & X \\ X^{-1}C & D^X \end{array} \right )
\left (\begin{array} {cc} E & FY^{-1} \\ Y & 0 \end{array} \right ) 
=
\left (\begin{array} {cc} XY & 0 \\ \ast & X^{-1}CFY^{-1} \end{array} \right ).
\]
Now $\tr W^S Q^T = \tr Z_1 + \tr Z_2^{Y^{-1}} = \tr Z_1 + \tr Z_2 = \tau$.
\end{proof}

\begin{lemma}               \label{ARADHERZOG5}
	Let $m \ge 2$.  
	Let $\Omega, \Psi$ be cyclic conjugacy  classes of $\SL(2m+1, K)$. 
	Then $\tr \Omega \Psi = K$.
\end{lemma}

\begin{proof} 
	If  $|K| \le 3$, then $\Omega$ and $\Psi$ are similarity classes, and we are already done by \ref{theorem-1}. So let $|K| \ge 4$. By \ref{ARADHERZOG3a}, there exist matrices 
	\[
	W = \left (\begin{array} {ccc} 0 & 0 & \Idm_m\\ 0 & \beta & x \\ C & d & Y  \end{array} \right ) \in \overline{\Omega}
	\ \mathrm{and} \ 
	Q = \left (\begin{array} {ccc} E & q & F\\ p & \gamma & 0 \\ \Idm_m & 0 & 0  \end{array} \right ) \in \overline{\Psi},
	\]
	where $C, Y, E, F \in \M(m,K)$, and $C$ and $F$ are cyclic. As in \ref{ARADHERZOG4}, we may assume that $CF = \lambda \Idm_1 \oplus M$ for some fixfree matrix $M$, where $M - \lambda \Idm_{m-1}$ is nonsingular. And conjugating $W$ by suitable matrix
	$\Idm_m \oplus \lambda \Idm_1 \oplus \Idm_m$ and $Q$  by a matrix $\Idm_m \oplus \mu \Idm_1 \oplus \Idm_m$, we may further assume that $W \in \Omega, Q \in \Psi$.
	
	Again by \ref{ARADHERZOG2}, $CF = X Z$, where 
	\begin{enumerate}
		\item $X$ and $Z$ are cyclic,
		\item $\det X = 1$,  $\tr X  = 0$ and $\tr Z  = \tau -\beta \gamma$.
	\end{enumerate}
	Let $T = \Idm_{m+1} \oplus X$.
	\[
	W^T Q = \left (\begin{array} {ccc} 0 & 0 & X\\ 0 & \beta & x \\ X^{-1}C & d & Y  \end{array} \right )  
	\left (\begin{array} {ccc} E & q & F\\ p & \gamma & 0 \\ \Idm_m & 0 & 0  \end{array} \right ) =
	\left (\begin{array} {ccc} X & 0 & 0\\ \ast & \beta \gamma & 0 \\ \ast & \ast & Z  \end{array} \right ).
	\]
\end{proof}

\begin{lemma}               \label{ARADHERZOG6}
	Let $K$ be finite.
	Let $\Omega, \Psi$ be nonscalar conjugacy  classes of $\SL(3, K)$. 
	Then $\tr \Omega \Psi = K$.
\end{lemma}

\begin{proof}
	By \ref{theorem-1}, we may assume that $\Omega, \Psi$ both are not similarity classes of $\GL(3, K)$. By \cite[THEOREM 2]{Newman-1982}, $\Omega$ and $\Psi$ have no irreducible elementary divisor. Hence there exist matrices
	\[
	\left (\begin{array} {cc} A  & b \\ 0 &  \lambda \end{array} \right ) \in \Omega,
	\left (\begin{array} {cc} \widetilde{A}  &  \widetilde{b}  \\ 0 & \mu  \end{array} \right ) \in \Psi,
	\]
	where $\mu(A) = (x-\lambda)^2, \mu(\widetilde{A}) = (x-\mu)^2$.
	Then 
	\[
	\left (\begin{array} {cc} X^{-1}  & 0 \\ 0 &  \det X \end{array} \right )
	\left (\begin{array} {cc} A  & b \\ 0 &  \lambda \end{array} \right )
	\left (\begin{array} {cc} X  & 0 \\ 0 &  \det X^{-1} \end{array} \right )
	\left (\begin{array} {cc} \widetilde{A}  &  \widetilde{b}  \\ 0 & \mu  \end{array} \right ) \in \Omega\Psi
	\]
	has trace $\tr X^{-1} A X \widetilde{A} + \lambda \mu$. We are done by lemma \ref{ARADHERZOG1}.
\end{proof}

\begin{remark}               \label{ARADHERZOG6a}
	Let $\Omega, \Psi \subseteq$ be  conjugacy classes of $\GL(2,3)$, where $\mu(\Omega) = x^2+1, \mu(\Psi) = x^2+x-1$. Then  $\Omega, \Psi$ and $-\Psi$ are the irreducible conjugacy classes of $\GL(2,3)$, and 
	\begin{enumerate}
		\item  $\Omega^2 = \Omega \cup \Idm_2 \cup -\Idm_2$,
		\item $\Psi^2 = \Omega \cup -\Idm_2 \cup U_2$,
		\item  $\Omega \Psi = -\Psi \cup \Psi \cup \Idm_1 \oplus -\Idm_1$,
	\end{enumerate}
	where $\mu(U_2) = (x-1)^2$.
	In particular, $\tr \Omega \Psi = \tr \Omega^2 = \GF(3)$.
\end{remark}

It can be shown $\tr \Omega \Psi = \GF(q)$ for arbitrary irreducible conjugacy classes 
$\GL(2,q)$.

\begin{lemma}               \label{ARADHERZOG7}
	Let $\Omega, \Psi$ be nonscalar conjugacy  classes of $\SL(4, 3)$. 
	Then $\tr \Omega \Psi = \GF(3)$.
\end{lemma}

\begin{proof}
	Again, we may assume that $\Omega$ and $\Psi$ have no irreducible elementary divisor. So there exist matrices $B, D, C, R$ such that
	\[
	W := \left (\begin{array} {cc} A  & \ast \\ 0 & D \end{array} \right ) \in \Omega,
	 Q:=  \left (\begin{array} {cc} R  &  \ast \\ 0 & T \end{array} \right )\in \Psi,
	\]
	where $A, D, R, T \in \GL(2, 3)$ are indecomposable.
	
	First let $\mu(W), \mu(Q) \in  \{(x^2+1)^2, (x^2 \pm x -1)^2\}$.
Now if $\det X = 1$, then 
	\[
	\left (\begin{array} {cc} A^X  & \ast \\ 0 & D \end{array} \right ) 
	\left (\begin{array} {cc} R  &  \ast\\ 0 & T\end{array} \right )
	= \left (\begin{array} {cc} A^X R & \ast \\ 0 & D T \end{array} \right ) 
	\in \Omega\Psi.
	\]
	By \ref{ARADHERZOG6a}, $\tr \Omega \Psi = \GF(3)$.	
	
	Next let $P^2$ and $Q^2$ be unipotent. Then we may assume that $D$ and $T$ are (upper) triangular. For $X \in \GL(2,3)$ let $Y =  \Idm_1 \oplus \det X \Idm_1$. Then 
	\[
	\left (\begin{array} {cc} A^X  & \ast \\ 0 & D^Y \end{array} \right ) 
	\left (\begin{array} {cc} R  &  \ast\\ 0 & T\end{array} \right )
	= \left (\begin{array} {cc} A^X R & \ast \\ 0 & D^Y T \end{array} \right ) 
	\in \Omega\Psi
	\]
	has trace $ \tr A^X R + \tr DT$. 
	
	Finally, let $Q^2$ be unipotent and $\mu(W) \in  \{(x^2+1)^2, (x^2 \pm x -1)^2\}$.
	If $\det X = \det Y = 1$, then 
	 \[
	 \left (\begin{array} {cc} A^X  & \ast \\ 0 & D^Y \end{array} \right ) 
	 \left (\begin{array} {cc} R  &  \ast\\ 0 & T\end{array} \right )
	 = \left (\begin{array} {cc} A^X R & \ast \\ 0 & D^Y T \end{array} \right )
	 \in \Omega\Psi.
	 \]
	 By lemma \ref{ARADHERZOG1}, $|\tr \{A^X R; X \in \SL(2, K) \}| = 2 = |\tr \{D^Y T; Y \in \SL(2, K) \}|$ so that 
	 $|\tr \Omega \Psi| = 3$.
	 We are done.
\end{proof}

It remains to consider the case that at least one conjugacy class is noncyclic.

For a $n \times n$ matrix $M$, let $\mr(M) := \min \{\rank(M - \lambda); \lambda \in \overline{K} \}$, where $\overline{K}$ is an algebraic closure of $K$, denote the minimal rank of $M$.

\begin{lemma}               \label{ARADHERZOG8}
 Let $M \in \GL(n, K)$, and let $A \in \M(m, K)$, where $m \le \mr(A), \frac{n}{2}$.
	Then $M$ is similar to a matrix
	\[
	\left (\begin{array} {cc} A & \ast \\ \ast & \ast \end{array} \right ).
	\]
\end{lemma}

\begin{proof}
	By \cite[THEOREM]{Sourour-1986}, $M$ is similar to a matrix
	\[
	\left (\begin{array} {ccc} 0_m & \Idm_m & \ast\\ \ast & \ast & \ast \\ \ast & \ast & \ast\end{array} \right ).
	\]
	Put $k = n-2m$. Then 
	\[
	\left (\begin{array} {ccc} \Idm_m & 0 & 0 \\ -A & \Idm_m & 0  \\ 0 & 0 & \Idm_{k}  \end{array} \right )
	\left (\begin{array} {ccc} 0_m & \Idm_m & \ast\\ \ast & \ast & \ast \\ \ast & \ast & \ast\end{array} \right )
	\left (\begin{array} {ccc} \Idm_m & 0 & 0 \\ A & \Idm_m & 0  \\ 0 & 0 & \Idm_{k}  \end{array} \right )
	= \left (\begin{array} {ccc} A & \Idm_m & \ast\\ \ast & \ast & \ast \\ \ast & \ast & \ast\end{array} \right ).
	\]
\end{proof}

\begin{lemma}               \label{ARADHERZOG9}
	Let $n \ge 4$. 
	Let $\Omega, \Psi$ be nonscalar conjugacy  classes of $\SL(n, K)$. 
	Then $\tr(\Omega \Psi) = K$.
\end{lemma}

\begin{proof}
	By \ref{ARADHERZOG4} and \ref{ARADHERZOG5}, we may assume that $\Psi$ is noncyclic.
	If $\Omega$ or $\Psi$ is a similarity class the assertion follows from \ref{ARADHERZOG3}. 
	We proceed as in the proof of \ref{ARADHERZOG3}.
	Let  $\Psi = \Psi_1 \oplus \Psi_2$, where $\Psi_1$ is cyclic and $2 \le \dim \Psi_1 \le \frac{n}{2}$.
	Put $r:= \dim \Psi_1$. Let
	\[
	W = \left (\begin{array} {cc} A & B \\ C & D \end{array} \right ) \in \Omega,
	Q = \left (\begin{array} {cc} R & 0 \\ 0 & S \end{array} \right ) \in \Psi,
	\]
	where $A \in \M(r, K)$ and $R \in \Psi_1$. 
	Clearly, $\mr(W) \ge \frac{n}{2}$ as $\Omega$ is not a similarity class.
	By \ref{ARADHERZOG8}, we can additionally assume that $A$ has an eigenvalue in $K$ of multiplicity one. 
	By theorem \ref{theorem-1} and lemma \ref{ARADHERZOG1}, $\tr \{ A^Y R; Y \in \GL(r, K)\} = K$.
	Since the centralizer  of $A$ contains matrices of arbitrary determinant, we have  $\tr \{ A^X R; X \in \SL(r, K)\} \supseteq \tr \{ A^Y R; Y \in \GL(r, K)\} = K$.
	Now $\Omega \Psi$ contains the matrices
	\[
	\left (\begin{array} {cc} A^X R & X^{-1}BS \\ CXR & DS \end{array} \right ) 
	=\left (\begin{array} {cc} X^{-1}  & 0 \\ 0 & \Idm_{n-r} \end{array} \right )
	\left (\begin{array} {cc} A & B \\ C & D \end{array} \right )
	\left (\begin{array} {cc} X  & 0 \\ 0 & \Idm_{n-r} \end{array} \right )
		\left (\begin{array} {cc} R & 0 \\ 0 & S \end{array} \right ),
	\]
	and we are done.
\end{proof}


\ifdraft{\listoflabels}{}

\end{document}

\typeout{get arXiv to do 4 passes: Label(s) may have changed. Rerun}